\documentclass{amsart}
\usepackage{amsfonts,amssymb,amsmath,amsthm,mathrsfs}
\usepackage{url}
\usepackage{enumerate}

\newtheorem{theorem}{Theorem}[section]
\newtheorem{lemma}[theorem]{Lemma}

\newtheorem{corollary}[theorem]{Corollary}
\theoremstyle{definition}

\newtheorem{remark}[theorem]{Remark}
\numberwithin{equation}{section}
\begin{document}

\title[Marcinkiewicz integral]
{A sparse domination for  the Marcinkiewicz integral with rough kernel and applications}
\author[X. Tao]{Xiangxing Tao}
\address{Department of Mathematics, School of Science, Zhejiang University of Science and Technology,
Hangzhou 310023, People's Republic of China }
\email{xxtao@zust.edu.cn}

\author[G. Hu]{Guoen Hu}
\address{Guoen Hu, Department of  Applied Mathematics, Zhengzhou Information Science and Technology Institute\\
Zhengzhou 450001,
P. R. China}
\email{guoenxx@163.com}
\thanks{The research of the first author was supported by NNSF of China under grant \#11771399 , and the  research of the second (corresponding author was supported by NNSF of China under grant $\#$11871108.}


\begin{abstract}
Let $\Omega$ be homogeneous of degree zero, have mean value zero and integrable on the unit sphere, and $\mu_{\Omega}$ be the higher-dimensional Marcinkiewicz
integral defined by
$$\mu_\Omega(f)(x)= \Big(\int_0^\infty\Big|\int_{|x-y|\leq t}\frac{\Omega(x-y)}{|x-y|^{n-1}}f(y)dy\Big|^2\frac{dt}{t^3}\Big)^{1/2}.
$$ In this paper, the authors establish a bilinear sparse domination for $\mu_{\Omega}$ under the assumption $\Omega\in L^{\infty}(S^{n-1})$. As applications, some quantitative weighted bounds for $\mu_{\Omega}$ are obtained.
\end{abstract}
\maketitle
\section{Introduction}

We will work on $\mathbb{R}^n$, $n\geq 2$. Let $\Omega$ be
homogeneous of degree zero, integrable and have mean value zero on
the unit sphere $S^{n-1}$. Define the Marcinkiewicz integral
operator $\mu_\Omega$ by
\begin{eqnarray*}\mu_\Omega(f)(x)= \Big(\int_0^\infty|F_{\Omega,\,
t}f(x)|^2\frac{dt}{t^3}\Big)^{1/2},
\end{eqnarray*}
where $$F_{\Omega,
t}f(x)= \int_{|x-y|\leq t}\frac{\Omega(x-y)}{|x-y|^{n-1}}f(y)dy $$ for
$f\in \mathcal{S}(\mathbb{R}^n)$.  Stein \cite{st}
proved that if $\Omega\in {\rm Lip}_{\rho}(S^{n-1})$ with
$\rho\in (0,\,1]$, then $\mu_\Omega$ is bounded on
$L^p(\mathbb{R}^n)$ for $p\in (1,\,2]$. Benedek, Calder\'on and
Panzon \cite{bcp} showed that the $L^p(\mathbb{R}^n)$ boundedness $(p\in
(1,\,\infty)$) of $\mu_\Omega$ holds true  under the
condition that $\Omega\in C^1(S^{n-1})$.  Walsh \cite{wal} proved
that for each $p\in (1,\,\infty)$, $\Omega\in L(\ln L)^{1/r}(\ln \ln
L)^{2(1-2/r')}(S^{n-1})$ is a sufficient condition such that
$\mu_\Omega$ is bounded on $L^{p}(\mathbb{R}^n)$, where
$r=\min\{p,\,p'\}$ and $p'=p/(p-1)$. Ding, Fan and Pan \cite{dfp}
proved that if $\Omega\in H^1(S^{n-1})$ (the Hardy space on
$S^{n-1}$), then $\mu_\Omega$ is bounded on
$L^p(\mathbb{R}^n)$  for all $p\in (1,\,\infty)$; Al-Salman,  Al-Qassem,   Cheng and Pan
\cite{aacp} proved that $\Omega\in L(\ln L)^{1/2}(S^{n-1})$ is a
sufficient condition such that $\mu_\Omega$ is bounded on
$L^p(\mathbb{R}^n)$ for all $p\in(1,\,\infty)$.  For other works about the operator $\mu_{\Omega}$, see \cite{a,   dfp,  fans} and the related references therein.

Let   $A_p(\mathbb{R}^n)$ $(p\in [1,\,\infty))$ be the weight function class of Muckenhoupt, that is, $$A_{p}(\mathbb{R}^n)=\{w\,\hbox{is\,\,nonnegative\,\,and\,\,locally\,\,integrable\,\,in}\,\,\mathbb{R}^n:\,[w]_{A_p}<\infty\},$$
where and in what follows, $$[w]_{A_p}:=\sup_{Q}\Big(\frac{1}{|Q|}\int_Qw(x)dx\Big)\Big(\frac{1}{|Q|}\int_{Q}w^{1-p'}(x)dx\Big)^{p-1},\,\,\,p\in (1,\,\infty),$$
and
$$[w]_{A_1}:=\sup_{x\in\mathbb{R}^d}\frac{Mw(x)}{w(x)},$$
with  $M$   the Hardy-Littlewood maximal operator. $[w]_{A_p}$ is called the $A_p$ constant of $w$, see  \cite[Chapter 9]{gra2} for the properties of $A_p(\mathbb{R}^n)$). In the last several years,  there has been significant  progeress in the study of  weighted bounds for rough singular integral operators. Hyt\"onen, Roncal and Tapiola \cite{hyta} considered the weighted  bounds of rough homogeneous singular integral operator defined by
$$T_{\Omega}f(x)={\rm p. \,v.}\int_{\mathbb{R}^n}\frac{\Omega(y')}{|y|^n}f(x-y)dy,$$
where $\Omega$ is homogeneous of degree zero, integrable on the unit sphere $S^{n-1}$ and has mean value zero. For $w\in \cup_{p>1}A_p(\mathbb{R}^n)$,
$[w]_{A_{\infty}}$ is the $A_{\infty}$ constant of $w$, defined by
$$[w]_{A_{\infty}}=\sup_{Q\subset \mathbb{R}^n}\frac{1}{w(Q)}\int_{Q}M(w\chi_Q)(x)dx,$$
see \cite{wil}.
By a quantitative weighted estimate for the Calder\'on-Zygmund operators,  approximation to the identity and interpolation, Hyt\"onen, Roncal  and Tapiola
(see Theorem 1.4 in \cite{hyta}) proved that
\begin{theorem}\label{dingli1.1} Let $\Omega$ be homogeneous of degree zero,   have mean value zero on $S^{n-1}$ and $\Omega\in L^{\infty}(S^{n-1})$.
Then for $p\in (1,\,\infty)$ and $w\in A_{p}(\mathbb{R}^n)$,
\begin{eqnarray}\label{eq:1.estimate}\|T_{\Omega}f\|_{L^p(\mathbb{R}^n,\,w)}&\lesssim & \|\Omega\|_{L^{\infty}(S^{n-1})}[w]_{A_p}^{\frac{1}{p}}\max\{[w]_{A_{\infty}}^{\frac{1}{p'}},\,[w^{1-p'}]_{A_{\infty}}^{\frac{1}{p}}\}\\
&&\times\max\{[w]_{A_{\infty}},\,[w^{1-p'}]_{A_{\infty}}\}
\|f\|_{L^p(\mathbb{R}^n,\,w)}.\nonumber
\end{eqnarray}
\end{theorem}
Conde-Alonso,   Culiuc,  Di Plinio and  Ou   \cite{hyta}  proved that for bounded function $f$ and $g$,  and $p\in (1,\,\infty)$,
\begin{eqnarray}\label{eq:1.ccdo}\Big|T_{\Omega}f(x)g(x)dx\Big|\lesssim p'\sup_{\mathcal{S}}\sum_{Q\in\mathcal{S}}\langle |f|\rangle_{Q}\langle |g|\rangle_{Q,\,p}|Q|,\end{eqnarray}
where the supremum is taken over all sparse family of cubes, $\langle |f|\rangle_Q$ denoted the mean value of $|f|$ on $Q$, and for $r\in (0,\,\infty)$, $\langle |f|\rangle_{Q,\,r}=
\big(\langle |f|^r\rangle_{Q}\big)^{1/r}.$ For a family of cubes $\mathcal{S}$, we say that $\mathcal{S}$ is $\eta$-sparse, $\eta\in (0,\,1)$, if  for each fixed $Q\in \mathcal{S}$, there exists a measurable subset $E_Q\subset Q$, such that $|E_Q|\geq \eta|Q|$ and $\{E_{Q}\}$ are pairwise disjoint. By  (\ref{eq:1.ccdo}) Conde-Alonso et al recovered the conclusion in Theorem \ref{dingli1.1}.
By some new estimates for bilinear sparse operators, Li, P\'erez, Rivera-Rios and  Roncal \cite{lpr} proved that
\begin{eqnarray*}\|T_{\Omega}f\|_{L^p(\mathbb{R}^n,\,w)}&\lesssim & \|\Omega\|_{L^{\infty}(S^{n-1})}[w]_{A_p}^{\frac{1}{p}}\max\{[w]_{A_{\infty}}^{\frac{1}{p'}},\,[w^{1-p'}]_{A_{\infty}}^{\frac{1}{p}}\}\\
&&\times\min\{[w]_{A_{\infty}},\,[w^{1-p'}]_{A_{\infty}}\}
\|f\|_{L^p(\mathbb{R}^n,\,w)},\nonumber
\end{eqnarray*}
which improved (\ref{eq:1.estimate}). Moreover, Li et al. \cite{lpr} proved that for any $w\in A_1(\mathbb{R}^n)$,
\begin{eqnarray*}\|T_{\Omega}f\|_{L^{1,\,\infty}(\mathbb{R}^n,\,w)}&\lesssim & \|\Omega\|_{L^{\infty}(S^{n-1})}[w]_{A_1}[w]_{A_{\infty}}\log ({\rm e}+[w]_{A_{\infty}})
\|f\|_{L^1(\mathbb{R}^n,\,w)}.
\end{eqnarray*}
They also established the weighted inequality of Coifman-Fefferman type that for $p\in [1,\,\infty)$ and $w\in A_{\infty}(\mathbb{R}^n)$, $$\|T_{\Omega}(f)\|_{L^p(\mathbb{R}^n,\,w)}\lesssim\|\Omega\|_{L^{\infty}(S^{n-1})}
[w]_{A_{\infty}}^2\|Mf\|_{L^p(\mathbb{R}^n,\,w)},$$

Let us return to the Marcinkiewicz integral. Ding, Fan and Pan \cite{dfp2} considered the boundedness on $L^p({\mathbb R}^n,\,w)$
with $w\in A_p(\mathbb{R}^n)$ for $\mu_{\Omega}$. They proved that
\begin{theorem}
Let $\Omega$ be homogeneous of degree zero, have mean value zero on $S^{n-1}$, and $\Omega\in L^{q}(S^{n-1})$ for some $q\in (1,\,\infty]$. Then for $p\in (q',\,\infty)$ and $w\in A_{p/q'}(\mathbb{R}^n)$ or $p\in (1,\,q)$ and $w^{1-p'}\in A_{p'/q'}(\mathbb{R}^n)$, then there exists a constant $C$ depending on $n$, $p$ and $w$, such that
$$\|\mu_{\Omega}(f)\|_{L^p(\mathbb{R}^n,\,w)}\leq C\|f\|_{L^p(\mathbb{R}^n,\,w)}.$$
\end{theorem}

Fan and Sato \cite{fans} established the weighted  weak type endpoint estimate for $\mu_{\Omega}$.
\begin{theorem}\label{dinglifans}Let $\Omega$ be homogeneous of degree zero, have mean value zero on $S^{n-1}$, and $\Omega\in L^{q}(S^{n-1})$ for some $q\in (1,\,\infty]$. Then for any $w$ with $w^{q'}\in A_1(\mathbb{R}^n)$, there exists a constant $C$ depending on $n$ and $w$, such that
$$\|\mu_{\Omega}(f)\|_{L^{1,\,\infty}(\mathbb{R}^n,\,w)}\leq C\|f\|_{L^1(\mathbb{R}^n,\,w)}.$$
\end{theorem}

Fairly recently, Hu and Qu \cite{huqu} established the following quantitative weighted bounds for $\mu_{\Omega}$.

\begin{theorem}\label{dingli1.2}
Let $\Omega$ be homogeneous of degree zero, have mean value zero on $S^{n-1}$, and $\Omega\in L^{\infty}(S^{n-1})$. Then for $p\in (1,\,\infty)$ and $w\in A_p(\mathbb{R}^n)$,
\begin{eqnarray*}\|\mu_{\Omega}(f)\|_{L^p(\mathbb{R}^n,\,w)}&\lesssim & \|\Omega\|_{L^{\infty}(S^{n-1})}[w]_{A_p}^{\frac{1}{p}}\big([w]_{A_{\infty}}^{(\frac{1}{2}-\frac{1}{p})_+}+
[w^{1-p'}]_{A_{\infty}}^{\frac{1}{p}}\big)\\
&&\times\max\{[w]_{A_{\infty}},\,[w^{1-p'}]_{A_{\infty}}\}
\|f\|_{L^p(\mathbb{R}^n,\,w)},
\end{eqnarray*}
where and in what follows, $(\frac{1}{2}-\frac{1}{p})_+=\max\{\frac{1}{2}-\frac{1}{p},\,0\}.$
\end{theorem}
Hu and Qu \cite{huqu} also proved that if $\Omega\in {\rm Lip}_{\rho}(S^{n-1})$ for some $\rho\in (0,\,1]$, then for bounded function $f$ with compact support, there exists a sparse family of cubes $\mathcal{S}$, such that for almost everywhere $x\in\mathbb{R}^n$,
$$\mu_{\Omega}(f)(x)\lesssim \Big(\sum_{Q\in\mathcal{S}}\langle |f|\rangle_Q^2\chi_{Q}(x)\Big)^{\frac{1}{2}},$$
see also \cite{ler4,lali} for the  sparse dominations of Littlewood-Paley square functions.
The proof of Theorem \ref{dingli1.2} follows from the idea of \cite{hyta}, together with the extension of Lerner's idea in \cite{ler3}, but does not involve the sparse domination of $\mu_{\Omega}$ when $\Omega\in L^{\infty}(S^{n-1})$.  As is well known, sparse dominations for classical operators in harmonic analysis have many applications  and are of independent interests. The purpose of this paper  is to prove that when $\Omega\in L^{\infty}(S^{n-1})$,  the Marcinkiewicz integral operator  enjoys a bilinear sparse domination similar to (\ref{eq:1.ccdo}). We remark that in this paper, we are very much motivated by
Lerner's works \cite{ler4,ler5}.
Our main result can be stated as follows.

\begin{theorem}\label{dingli1.3}
Let $\Omega$ be homogeneous of degree zero, have mean value zero on $S^{n-1}$, and $\Omega\in L^{\infty}(S^{n-1})$. Then for each bounded function $f$ with compact support, there exists a $\frac{1}{2}\frac{1}{3^n}$-sparse family of cubes $\mathcal{S}$, such that for any $r\in (1,\,2]$, $s\in[1,\,2]$, and nonnegative function $g\in L_{{\rm loc}}^r(\mathbb{R}^n)$,
\begin{eqnarray}\label{eq:1.3}
\int_{\mathbb{R}^n}\big(\mu_{\Omega}(f)(x)\big)^sg(x)dx\lesssim r'^{\frac{s}{2}}\|\Omega\|_{L^{\infty}(S^{n-1})}^s\sum_{Q\in\mathcal{S}}\langle |f|\rangle_{Q}^s\langle |g|\rangle_{Q,\,r}|Q|.
\end{eqnarray}
\end{theorem}

As   applications of Theorem \ref{dingli1.3}, we obtain the following quantitative weighted bounds for $\mu_{\Omega}$.
 \begin{theorem}\label{dingli1.4} Let $\Omega$ be homogeneous of degree zero,   have mean value zero on $S^{n-1}$ and $\Omega\in L^{\infty}(S^{n-1})$.
Then
\begin{itemize}
\item[\rm (i)]for $p\in (1,\,\infty)$ and $w\in A_{p}(\mathbb{R}^n)$,
\begin{eqnarray}\label{eq:1.4}\|\mu_{\Omega}(f)\|_{L^p(\mathbb{R}^n,\,w)}&\lesssim & \|\Omega\|_{L^{\infty}(S^{n-1})}[w]_{A_p}^{\frac{1}{p}}\big([w]_{A_{\infty}}^{(\frac{1}{2}-\frac{1}{p})_+}+
[w^{1-p'}]_{A_{\infty}}^{\frac{1}{p}}\big)\\
&&\times[w]_{A_{\infty}}^{\frac{1}{2}}
\|f\|_{L^p(\mathbb{R}^n,\,w)}.\nonumber
\end{eqnarray}
\item[\rm (ii)] For $w\in A_1(\mathbb{R}^n)$,
\begin{eqnarray}\label{eq:1.5}&&\|\mu_{\Omega}(f)\|_{L^{1,\,\infty}(\mathbb{R}^n)}\lesssim \|\Omega\|_{L^{\infty}(S^{n-1})}[w]_{A_1}[w]_{A_{\infty}}^{\frac{1}{2}}\log ({\rm e}+[w]_{A_{\infty}})\|f\|_{L^1(\mathbb{R}^n,\,w)}.\end{eqnarray}
\end{itemize}
\end{theorem}
\begin{remark} Estimate (\ref{eq:1.4}) improves the conclusion of Theorem \ref{dingli1.2}.  and estimate (\ref{eq:1.5}) is a quantitative version of Theorem \ref{dinglifans} when $\Omega\in L^{\infty}(S^{n-1})$.
\end{remark}
We also have the following weighted norm inequality of Coifman-Fefferman type for $\mu_{\Omega}$.
\begin{theorem}\label{dingli1.5} Let $\Omega$ be homogeneous of degree zero,   have mean value zero on $S^{n-1}$ and $\Omega\in L^{\infty}(S^{n-1})$.
Then for $w\in A_{\infty}(\mathbb{R}^n)$,
$$\|\mu_{\Omega}(f)\|_{L^p(\mathbb{R}^n,\,w)}\lesssim\|\Omega\|_{L^{\infty}(S^{n-1})}\Bigg\{ \begin{array}{ll}\displaystyle
[w]_{A_{\infty}}^{\frac{1}{p}+\frac{1}{2}}
\|Mf\|_{L^p(\mathbb{R}^n,\,w)},\,&\hbox{if}\,\,p\in [1,\,2],\\
\,[w]_{A_{\infty}}\|Mf\|_{L^p(\mathbb{R}^n,\,w)},\,&\hbox{if}\,\,p\in (2,\infty),\end{array}$$
provided that $f$ is a bounded function with compact support.
\end{theorem}
We make some conventions. In what follows, $C$ always denotes a
positive constant that is independent of the main parameters
involved but whose value may differ from line to line. We use the
symbol $A\lesssim B$ to denote that there exists a positive constant
$C$ such that $A\le CB$.   For a set $E\subset\mathbb{R}^n$,
$\chi_E$ denotes its characteristic function. For a cube $Q$, we use $\ell(Q)$ to denote the side lenth of $Q$.

\section{Proof of Theorem \ref{dingli1.3}}
We begin by recalling some definitions and a basic lemma from \cite{lernaza}. Given
a cube $Q_0$, we denote by $\mathcal{D}(Q_0)$ the   set of dyadic cubes with respect to $Q_0$, that is, the cubes from $\mathcal{D}(Q_0)$ are formed by repeating subdivision of $Q_0$ and each of descendants into $2^n$ congruent subcubes.

Let $\mathscr{D}$ be a  collection of cubes in $\mathbb{R}^n$. We say  that $\mathscr{D}$  a dyadic lattice if
\begin{itemize}
\item[\rm (i)]if $Q\in\mathscr{D}$, then $\mathcal{D}(Q)\subset \mathscr{D}$;
\item[\rm (ii)] every two cubes $Q',\, Q''\in\mathscr{D}$ have a common ancestor,  that is, there
exists $Q\in\mathscr{D}$ such that $Q¡ä,\,Q¡ä¡ä \in \mathcal{D}(Q)$;
\item[\rm (iii)] for every compact set $K\subset \mathbb{R}^n$,  there exists a cube $Q\in\mathscr{D}$
containing $K$.
\end{itemize}
\begin{lemma}\label{lem2.1} {\rm (The Three Lattices Theorem)} For every dyadic lattice
$\mathscr{D}$, there exist $3^n$ dyadic lattices $\mathscr{D}_1$,\,\dots,\,$\mathscr{D}_{3^n}$  such that
$$\{3Q : Q \in\mathscr{D}\} = \cup_{j=1}^{3n}\mathscr{D}_j$$
and for every cube $Q \in \mathscr{D}$ and $j = 1,\,\dots,\, 3^n$, there exists a unique cube
$R \in  \mathscr{D}_j$ of side length $\ell(R)=3\ell(Q)$ containing $Q$.
\end{lemma}

For $t\in [1,\,2]$ and $j\in\mathbb{Z}$, set
\begin{eqnarray}\label{eq:kernel-t-j}
K^j_t(x)=\frac{1}{2^j}\frac{\Omega(x)}{|x|^{n-1}}\chi_{\{2^{j-1}t<|x|\leq 2^jt\}}(x).\end{eqnarray}
Let
\begin{eqnarray}\label{eq1.tilde}\widetilde{\mu}_{\Omega}(f)(x)=\Big(\int^2_1\sum_{j\in\mathbb{Z}}\big|F_{j}f(x,\,t)\big|^2dt\Big)^{1/2},\end{eqnarray}with
$$F_jf(x,\,t)=\int_{\mathbb{R}^n}K^j_t(x-y)f(y)dy.$$
A trivial computation leads to that
\begin{eqnarray}\label{eq:equi}
\mu_{\Omega}(f)(x)\approx \widetilde{\mu}_{\Omega}(f)(x).
\end{eqnarray}
Observe that
\begin{eqnarray}\label{eq2.4}\widetilde{\mu}_{\Omega}(f)(x)=\Big(\int^{\infty}_0|\widetilde{F}_t(f)(x)|^2\frac{dt}{t^3}\Big)^{\frac{1}{2}},\end{eqnarray}
with
$$\widetilde{F}_t(f)(x)=\int_{t/2<|x-y|\leq t}\frac{\Omega(x-y)}{|x-y|^{n-1}}f(y)dy.
$$

 Let $\phi\in
C^{\infty}_0(\mathbb{R}^n)$ be a nonnegative function such that
$\int_{\mathbb{R}^n}\phi(x)dx=1$, ${\rm
supp}\,\phi\subset\{x:\,|x|\leq 1/4\}$. For $l\in \mathbb{Z}$, let
$\phi_l(y)=2^{-nl}\phi(2^{-l}y)$.
Define the operator $\widetilde{\mu}_{\Omega}^l$ by
\begin{eqnarray*}\widetilde{\mu}_{\Omega}^l(f)(x)=\Big(\int^2_1\sum_{j\in\mathbb{Z}}\big|F_{j}^lf(x,\,t)\big|^2
dt\Big)^{1/2},\end{eqnarray*}where
$$F_{j}^lf(x,\,t)=\int_{\mathbb{R}^n}K^j_t*\phi_{j-l}(x-y)f(y)\,dy.
$$
It was proved in \cite{huqu}
that for some $\kappa\in (0,\,1)$.
\begin{eqnarray}\label{eq2.5}
\|\widetilde{\mu}_{\Omega}(f)-\widetilde{\mu}_{\Omega}^{l}(f)\|_{L^2(\mathbb{R}^n)}^2
\lesssim2^{-2\kappa l}\|\Omega\|_{L^{\infty}(S^{n-1})}\|f\|_{L^2(\mathbb{R}^n)}^2.
\end{eqnarray}
Also, for
$l,\,k\in\mathbb{N}$, $R>0$ and  $y\in \mathbb{R}^n$ with $|y|<R/4$,
\begin{eqnarray}\label{eq2.6}&&\sum_{j\in\mathbb{Z}}\sup_{2^kR<|x|\leq 2^{k+1}R}\sup_{t\in [1,\,2]}\big|K^j_{t}*\phi_{j-l}(x+y)-K^j_{t}*\phi_{j-l}(x)\big|\\
&&\quad\lesssim \|\Omega\|_{L^{\infty}(S^{n-1})}\frac{1}{(2^kR)^n}\min\{1,\,2^l\frac{|y|}{2^kR}\},\nonumber\end{eqnarray}
see \cite[Lemma 2.2]{huqu}.
\begin{lemma}\label{lem2.2} Let $\Omega$ be homogeneous of degree zero and have mean value zero.
Suppose that $\Omega\in L^{\infty}(S^{n-1})$. Then ${\widetilde{\mu}_{\Omega}^l}$ is bounded from $L^{1}(\mathbb{R}^n)$ to $L^{1,\,\infty}(\mathbb{R}^n)$ with bound less than $Cl^{\frac{1}{2}}\|\Omega\|_{L^{\infty}(S^{n-1})}$.
\end{lemma}
\begin{proof} We modify the argument used in the proof of Lemma 2.3 in \cite{huqu}, in which it was proved that ${\widetilde{\mu}_{\Omega}^l}$ is bounded from $L^{1}(\mathbb{R}^n)$ to $L^{1,\,\infty}(\mathbb{R}^n)$ with bound less than $Cl\|\Omega\|_{L^{\infty}(S^{n-1})}$. The argument here involves some refined kernel estimates. Our goal is to prove that  for any $\lambda>0$,
\begin{eqnarray}\label{eq:weak type(1,1)}
\big|\big\{x\in\mathbb{R}^n:\,\widetilde{\mu}_{\Omega}^l(f)(x)>\lambda\big\}\big|\lesssim l^{\frac{1}{2}}\lambda^{-1}\|\Omega\|_{L^{\infty}(S^{n-1})}\|f\|_{L^1(\mathbb{R}^n)}.\end{eqnarray}
Without loss of generality, we assume that $\|\Omega\|_{L^{\infty}(S^{n-1})}=1$. For each fixed $\lambda>0$,  applying the Calder\'on-Zygmund decomposition to $|f|$ at level $\lambda$,
we obtain a sequence of cubes $\{Q_i\}$ with disjoint interiors, such that
$$\lambda<\frac{1}{|Q_i|}\int_{Q_i}|f(y)|dy\le 2^n\lambda,$$
and  $|f(y)|\lesssim \lambda$ for a. e. $y\in\mathbb{R}^n\backslash\big(\cup_{i}Q_i\big).$ Set
$$g(y)=f(y)\chi_{\mathbb{R}^n\backslash\cup_iQ_i}(y)+\sum_{i}\langle f\rangle_{Q_i}\chi_{Q_i}(y),$$
$$b(y)=\sum_{i}b_i(y),\,\,\hbox{with}\,\,
b_i(y)=\big(f(y)-\langle f\rangle_{Q_i}\big)\chi_{Q_i}(y).$$
By \eqref{eq2.5} and the $L^2(\mathbb{R}^n)$ boundedness of $\widetilde{\mu}_{\Omega}$, we know that $\widetilde{\mu}_{\Omega}^l$ is also bounded on $L^2(\mathbb{R}^n)$ with bound independent of $l$. Therefore,
$$|\{x\in\mathbb{R}^n:\,\widetilde{\mu}_{\Omega}^l(g)(x)>\lambda/2\}|\lesssim\lambda^{-2}
\|\widetilde{\mu}_{\Omega}^lg\|_{L^2(\mathbb{R}^n)}^2\lesssim\lambda^{-1}\|f\|_{L^1(\mathbb{R}^n)}.$$
Let $E_{\lambda}=\cup_i4nQ_i$. It is obvious that $|E_{\lambda}|\lesssim \lambda^{-1}\|f\|_{L^1(\mathbb{R}^n)}$. The proof of \eqref{eq:weak type(1,1)} is now reduced to proving that
\begin{eqnarray}\label{eq2.8}
|\{x\in\mathbb{R}^n\backslash E_{\lambda}:\,\widetilde{\mu}_{\Omega}^l(b)(x)>\lambda/2\}|\lesssim l^{\frac{1}{2}}\lambda^{-1}\|f\|_{L^1(\mathbb{R}^n)}.
\end{eqnarray}

We now prove \eqref{eq2.8}. Let
$$U_l(b)(x)=\int^2_1\sum_{j\in\mathbb{Z}}\big|F_{j}^lb(x,\,t)\big|
dt.
$$
For each fixed $i$, let $x_i$ be the center of $Q_i$. A trivial computation involving (\ref{eq2.6}) shows that for $y\in Q_i$ and $t\in [1,\,2]$,
\begin{eqnarray*}
&&\sum_j\int_{\mathbb{R}^n\backslash 4nQ_i}|K_{t}^j*\phi_{j-l}(x-y)-K_t^j*\phi_{j-l}(x-x_i)|dx\\
&&\quad\lesssim \sum_j\sum_{k=2}^{\infty}\sup_{x\in 2^{k+1}Q_i\backslash 2^{k}Q_i}|K_t^j*\phi_{j-l}(x-y)-K_t^j*\phi_{j-l}(x-x_i)||2^k Q_i|\\
&&\quad\lesssim\sum_{k=2}^{\infty}\min\{1,\,2^{l-k}\}\lesssim l.
\end{eqnarray*}
Therefore,
\begin{eqnarray}\label{eq2.9}
&&|\{x\in\mathbb{R}^n\backslash E_{\lambda}:\,U_l(b)(x)>l^{\frac{1}{2}}\lambda/2\}|\\
&&\quad\le 2l^{-\frac{1}{2}}\lambda^{-1}\sum_i\sum_j\int^2_1\int_{\mathbb{R}^n\backslash E_{\lambda}}|F_{j}^lb_i(x,\,t)|dxdt\nonumber\\
&&\quad\lesssim l^{\frac{1}{2}}\lambda^{-1}\|f\|_{L^1(\mathbb{R}^n)}.\nonumber
\end{eqnarray}
Observe that for each fixed $j\in\mathbb{Z}$,
$$\sup_{t\in [1,\,2]}\Big|K^j_{t}*\phi_{j-l}(x)\Big|\lesssim \int_{\mathbb{R}^n}\widetilde{K}^j(z)\big|\phi_{j-l}
(x-z)\big|dz\lesssim 2^{-jn}\chi_{\{x:\,2^{j-4}\leq |x|\leq 2^{j+4}\}}(x),$$
with $\widetilde{K}^j(z)=|z|^{-n}|\Omega(z)|\chi_{\{2^{j-2}\leq |z|\leq 2^{j+2}\}}(z).$ It then follows   that
$$\sup_{j\in\mathbb{Z}}\sup_{t\in [1,\,2]}\big|K^j_{t}*\phi_{j-l}*b(x)\big|\lesssim Mb(x).$$
This, in turn , implies that
\begin{eqnarray}\label{eq2.10}&&|\{x\in\mathbb{R}^n:\,\sup_{j\in\mathbb{Z}}\sup_{t\in [1,\,2]}\big|K^j_{t}*\phi_{j-l}*b(x)\big|>l^{-\frac{1}{2}}\lambda/2\}|\lesssim l^{\frac{1}{2}}\lambda^{-1}\|f\|_{L^1(\mathbb{R}^n)}.
\end{eqnarray}
Combining estimates (\ref{eq2.9}) and (\ref{eq2.10}) leads to that
\begin{eqnarray*}&&|\{x\in\mathbb{R}^n\backslash E_{\lambda}:\,\widetilde{\mu}_{\Omega}^l(b)(x)>\lambda/2\}|\\
&&\leq|\{x\in\mathbb{R}^n\backslash E_{\lambda}:\,U_l(b)(x)>l^{\frac{1}{2}}\lambda/2\}|\\
&&\quad+|\{x\in\mathbb{R}^n:\,\sup_{j\in\mathbb{Z}}\sup_{t\in [1,\,2]}\big|K^j_{t}*\phi_{j-l}*b(x)\big|>l^{-\frac{1}{2}}\lambda/2\}|\\
&&\quad\lesssim \lambda^{-1}l^{\frac{1}{2}}\|f\|_{L^1(\mathbb{R}^n)}.
\end{eqnarray*}
This leads to (\ref{eq2.8}) and completes the proof of Lemma 2.2.
\end{proof}
\begin{lemma}\label{lem2.3} Let $\Omega$ be homogeneous of degree zero and have mean value zero.
Suppose that $\Omega\in L^{\infty}(S^{n-1})$. Let $\mathcal{M}_{\widetilde{\mu}_{\Omega}^l}$ be the grand maximal operator  defined by
$$\mathcal{M}_{\widetilde{\mu}_{\Omega}^l}f(x)=\sup_{Q\ni x}\|\widetilde{\mu}_{\Omega}^l(f\chi_{\mathbb{R}^n\backslash 3Q})\|_{L^{\infty}(Q)}$$
Then $\mathcal{M}_{\widetilde{\mu}_{\Omega}^l}$ is bounded from $L^{1}(\mathbb{R}^n)$ to $L^{1,\,\infty}(\mathbb{R}^n)$ with bound $Cl^{\frac{1}{2}}\|\Omega\|_{L^{\infty}(S^{n-1})}$.
\end{lemma}
\begin{proof}Again we assume that $\|\Omega\|_{L^{\infty}(S^{n-1})}=1$.
Let $x\in\mathbb{R}^n$  and $Q\subset \mathbb{R}^n$ be a cube containing $x$. Let$R=2{\rm diam}Q$ and denote by $B_x$ the closed ball centered
at $x$ with  radius $R$. Then $3Q\subset B_x$. For each $\xi\in  Q$,   we can write
\begin{eqnarray*}|\widetilde{\mu}_{\Omega}^l(f\chi_{R^n\backslash 3Q})(\xi)| &\leq & |\widetilde{\mu}_{\Omega}^l(f\chi_{\mathbb{R}^n\backslash B_x})(\xi)- \widetilde{\mu}_{\Omega}^l(f\chi_{\mathbb{R}^n\backslash B_x})(x)
|\\
&&+ \widetilde{\mu}_{\Omega}^l(f\chi_{B_x\backslash 3Q})(\xi) + \widetilde{\mu}_{\Omega}^l(f\chi_{\mathbb{R}^n\backslash B_x})(x).
\end{eqnarray*}
It was proved in \cite{huqu} that
\begin{eqnarray}\label{eq2.11}\sup_{\xi\in Q}\widetilde{\mu}_{\Omega}^l(f\chi_{B_x\backslash 3Q})(\xi)\lesssim Mf(x),
\end{eqnarray}
and
\begin{eqnarray}\widetilde{\mu}_{\Omega}^l(f\chi_{\mathbb{R}^n\backslash B_x})(x)\lesssim Mf(x)+\widetilde{\mu}_{\Omega}^l(f)(x).\end{eqnarray}
Now write
\begin{eqnarray*}
&&|\widetilde{\mu}_{\Omega}^l(f\chi_{\mathbb{R}^n\backslash B_x})(\xi)- \widetilde{\mu}_{\Omega}^l(f\chi_{\mathbb{R}^n\backslash B_x})(x)
|\\
&&\quad\leq \Big(\int^2_1\sum_{j\in\mathbb{Z}}\Big|\int_{\mathbb{R}^n}R_{t}^{j,\,l}(x;\,y,\,\xi)f(y)\chi_{\mathbb{R}^n\backslash B_x}(y)dy\Big|^2dt\Big)^{\frac{1}{2}},
\end{eqnarray*}
where
$$R_{t}^{j,\,l}(x;\,y,\,\xi)=|K_t^j*\phi_{l-j}(x-y)-K_t^j*\phi_{l-j}(\xi-y)|.$$
Obviously, for $x,\,\xi\in Q$,
$R_{t}^{j,\,l}(x;\,y,\,\xi)\chi_{\mathbb{R}^n\backslash Q_x}(y)\not =0$ only if $y\in B(x,\,2^{j+5})$, since
$${\rm supp}\,K_t^j*\phi_{l-j}\subset \{z:\, 2^{j-3}\leq |z|\leq 2^{j+3}\}.$$
A trivial computation shows that
$$|R_{t}^{j,\,l}(x;\,y,\,\xi)|\lesssim 2^{-jn}\min\{1,\,2^{l-j}|x-\xi|\}.
$$
This, along with H\"older's inequality, gives us that¡¡
\begin{eqnarray}\label{eq2.13}
&&\sup_{t\in [1,\,2]}\sum_{j\in\mathbb{Z}}\Big|\int_{\mathbb{R}^n}R_{t}^{j,\,l}(x;\,y,\,\xi)f(y)\chi_{\mathbb{R}^n\backslash B_x}(y)dy\Big|^2
\\
&&\quad\leq \sup_{t\in [1,\,2]}\sum_{j:2^j>\frac{R}{16}}\Big|\int_{B(x,\,2^{j+5})}R_{t}^{j,\,l}(x;\,y,\,\xi)f(y)\chi_{\mathbb{R}^n\backslash B_x}(y)dy\Big|^2\nonumber\\
&&\quad\lesssim \sum_{j:\,2^j>\frac{R}{16}}\big(\min\{1,\,\frac{2^lR}{2^j}\}\big)^2\big(Mf(x)\big)^2\lesssim l\big(Mf(x)\big)^2.\nonumber
\end{eqnarray}
Collecting  estimates (\ref{eq2.11})-(\ref{eq2.13}) yields
$$\mathcal{M}_{\widetilde{\mu}^l_{\Omega}}f(x)\lesssim l^{\frac{1}{2}}Mf(x)+\widetilde{\mu}_{\Omega}^l(f)(x).$$
This, via Lemma \ref{lem2.2}, leads to our desired conclusion.
\end{proof}

Given an operator $T$, define the maximal operator ${M}_{\lambda,\,T}$ by
$${M}_{\lambda,\,T}f(x)=\sup_{Q\ni x}\Big(T(f\chi_{\mathbb{R}^n\backslash 3Q})\chi_{Q}\Big)^*(\lambda |Q|),\,\,(0<\lambda<1),$$
where the supremum is taken over all cubes $Q\subset \mathbb{R}^n$ containing $x$, and $h^*$ denotes the non-increasing rearrangement of $h$.
The operator ${M}_{\lambda,\,T}$ was introduced by Lerner \cite{ler5} and is useful in the study of weighted bounds for rough operators, see \cite{ler5,hulai}.
\begin{theorem}\label{dingli2.1}Let $\Omega$ be homogeneous of degree zero and have mean value zero.
Suppose that $\Omega\in L^{\infty}(S^{n-1})$. Then for $\lambda\in (0,\,1)$,
\begin{eqnarray}\label{eq2.14}\|M_{\lambda,\,\widetilde{\mu}_{\Omega}}f\|_{L^{1,\,\infty}(\mathbb{R}^n)}\lesssim \|\Omega\|_{L^{\infty}(S^{n-1})}\big(1+\log^{\frac{1}{2}} \big(\frac{1}{\lambda}\big)\big)\|f\|_{L^{1}(\mathbb{R}^n)}.\end{eqnarray}
\end{theorem}

\begin{proof}
We mimic the proof of Lemma 2.3 in \cite{ler5}. Let $\mathscr{D}$ and $\mathscr{D}'$ be two dyadic lattices  and $\mathcal{F}$ be a finite family of cubes $Q$ from $\mathscr{D}$ such that $3Q\in \mathscr{D}'$. Set
$$M_{\lambda,\,\widetilde{\mu}
_{\Omega}}^{\mathcal{F}}f(x)=\Bigg\{ \begin{array}{ll}
\max_{Q\ni x,\,Q\in \mathcal{F}}\big(\widetilde{\mu}_{\Omega}(f\chi_{\mathbb{R}^n\backslash 3Q})\chi_{Q}\big)^*(\lambda |Q|),\,&x\in \cup_{Q\in\mathcal{F}}Q\\
0,\,&\hbox{otherwise}.\end{array}$$
As it was pointed out in \cite{ler5}, it suffices to prove (\ref{eq2.14}) with $M_{\lambda,\widetilde{\mu}_{\Omega}}$ replaced by $M^{\mathcal{F}}_{\lambda,\widetilde{\mu}_{\Omega}}$.

We assume that $\|\Omega\|_{L^{\infty}(S^{n-1})}=1$. Let $M^{\mathscr{D'}}$ be the mximal operator defined by
$$M^{\mathscr{D'}}f(x)=\sup_{Q\ni x,\,Q\in\mathscr{D'}}\langle |f|\rangle_Q.$$
Now let $f\in L^1(\mathbb{R}^n)$, $\alpha>0$. Decompose $\{x\in\mathbb{R}^n:\,M^{\mathscr{D}'}f(x)>\lambda^{-1}\alpha\}$ as $\{x\in\mathbb{R}^n:\,M^{\mathscr{D}'}f(x)>\lambda^{-1}\alpha\}=\cup_{P\in\mathcal{P}} P$, with $P$ the maximal cubes in $\mathscr{D}'$ such that $\frac{1}{|P|}\int_P|f(y)|dy>\lambda^{-1}\alpha$. For each $P\in\mathcal{P}$, let
$$b_P(y)=(f(y)-\langle f\rangle_P)\chi_{P}(y).$$ For $j\in \mathbb{Z}$, set $B_j(x)=\sup_{P\in\mathcal{P}:\,|P|=2^{nj}}b_P(x).$
We decompose $f$ as $f=g+b$ with $b=\sum_{P\in\mathcal{P}}b_P$ and $g=f-b.$ Then $\|g\|_{L^{\infty}(\mathbb{R}^n)}\lesssim \lambda^{-1}\alpha$.
Let $l\in\mathbb{N}$ such that $\log(\frac{1}{\lambda})+1<\kappa l\leq\log(\frac{1}{\lambda})+2$, with $\kappa$ the positive constant in (\ref{eq2.5}).
A straightforward composition involving Lemma 2.5 in \cite{ler5} and (\ref{eq2.5})  leads to that
\begin{eqnarray}\label{eq2.15}
|\{x\in\mathbb{R}^n:\,M_{\lambda/4, \widetilde{\mu}_{\Omega}-\widetilde{\mu}_{\Omega}^{2l}}^{\mathcal{F}}g(x)>\alpha\}|&\lesssim&\alpha^{-2}\lambda^{-1}
\|\widetilde{\mu}_{\Omega}(g)-\widetilde{\mu}_{\Omega}^{2l}(g)\|_{L^2(\mathbb{R}^n)}^2\\
&\lesssim&\alpha^{-2}\lambda^{-1}2^{-2\kappa l}\|g\|_{L^2(\mathbb{R}^n)}^2\nonumber\\
&\lesssim&\alpha^{-1}\|f\|_{L^1(\mathbb{R}^n)}.\nonumber
\end{eqnarray}
Set $E=\cup_{P\in\mathcal{P}}9P$ and $E^*=\{x\in\mathbb{R}^n:\,M^{\mathscr{D}}\chi_{E}(x)>\frac{\lambda}{32}\}$. We then have that
$$|E^*|\lesssim \lambda^{-1}|E|\lesssim\frac{1}{\alpha}\|f\|_{L^1(\mathbb{R}^n)}.
$$
Note that
\begin{eqnarray}&&|\{x\in\mathbb{R}^n:\, M^{\mathcal{F}}_{\lambda/4, \widetilde{\mu}_{\Omega}-\widetilde{\mu}_{\Omega}^{2l}}b(x)>\alpha/2\}|\\
&&\quad\lesssim |\{x\in\mathbb{R}^n\backslash E^*:\, M_{\lambda/8,\widetilde{\mu}_{\Omega}}^{\mathcal{F}}b(x)>\alpha/4\}|+|E^*|\nonumber\\
&&\qquad+|\{x\in\mathbb{R}^n:\, M_{\lambda/8,\widetilde{\mu}_{\Omega}^{2l}}^{\mathcal{F}}b(x)>\alpha/4\}|.\nonumber
\end{eqnarray}
Also, we have by Lemma \ref{lem2.1} that
\begin{eqnarray}\label{eq2.17}&&|\{x\in\mathbb{R}^n:\, M_{\lambda/8,\widetilde{\mu}_{\Omega}^{2l}}^{\mathcal{F}}b(x)>\alpha/4\}|\lesssim\alpha^{-1}l^{\frac{1}{2}}\|b\|_{L^1(\mathbb{R}^n)}\lesssim
\alpha^{-1}l^{\frac{1}{2}}\|f\|_{L^1(\mathbb{R}^n)}.
\end{eqnarray}
If we can prove that
\begin{eqnarray}\label{eq2.18}|\{x\in\mathbb{R}^n\backslash E^*:\, M_{\lambda/8, \widetilde{\mu}_{\Omega}}^{\mathcal{F}}b(x)>\alpha/4\}|\lesssim
\alpha^{-1}l^{\frac{1}{2}}\|f\|_{L^1(\mathbb{R}^n)},
\end{eqnarray}
then  by inequalities (\ref{eq2.15})--(\ref{eq2.18}), we have that
\begin{eqnarray*}|\{x\in\mathbb{R}^n:\, M_{\lambda/2, \widetilde{\mu}_{\Omega}-\widetilde{\mu}_{\Omega}^{2l}}^{\mathcal{F}}f(x)>\alpha/4\}|\lesssim\alpha^{-1}l^{\frac{1}{2}}\|b\|_{L^1(\mathbb{R}^n)}\lesssim
\alpha^{-1}l^{\frac{1}{2}}\|f\|_{L^1(\mathbb{R}^n)},
\end{eqnarray*}
This along with Lemma \ref{lem2.3}  and the fact that
\begin{eqnarray*}\|M_{\lambda, \widetilde{\mu}_{\Omega}}^{\mathcal{F}}f\|_{L^{1,\,\infty}(\mathbb{R}^n)}\leq \|M^{\mathcal{F}}_{\lambda/2, \widetilde{\mu}_{\Omega}-\widetilde{\mu}_{\Omega}^{2l}}f\|_{L^{1,\,\infty}(\mathbb{R}^n)}
+\|M^{\mathcal{F}}_{\lambda/2, \widetilde{\mu}_{\Omega}^{2l}}f\|_{L^{1,\,\infty}(\mathbb{R}^n)},
\end{eqnarray*}
leads to (2.14).

We now prove (\ref{eq2.18}). We will employ the ideas of Fan and Sato \cite{fans}. Let $h(r)=r\chi_{(1/2,\,1]}(r)$,  and
$$L_{j,\,t}(x)=\frac{\Omega(x)}{|x|^{n}}h(\frac{|x|}{t})\chi_{(1,\,2]}(2^{-j}t).
$$
Let $m\in\mathbb{N}$ which will be chosen later. For each $Q\in\mathcal{F}$, $x\in Q$ and $x\in\mathbb{R}^n\backslash E$, we have that
\begin{eqnarray*}&&\Big(\widetilde{\mu}_{\Omega}(b\chi_{\mathbb{R}^n\backslash 3Q})\Big)^*(\frac{\lambda|Q|}{8})\\
&&\quad=\Big[\Big(\int^{\infty}_0\Big|\sum_{s=1}^{\infty}\sum_{j\in\mathbb{Z}}L_{j,\,t}*\big(B_{j-s}\chi_{\mathbb{R}^n\backslash 3Q}\big)\big|^2\frac{dt}{t}\Big)^{\frac{1}{2}}\Big]^*(\frac{\lambda|Q|}{8})\\
&&\quad\le \Big[\Big(\int^{\infty}_0\Big|\sum_{s=1}^m\sum_{j\in\mathbb{Z}}L_{j,\,t}*\big(B_{j-s}\chi_{\mathbb{R}^n\backslash 3Q}\big)\big|^2\frac{dt}{t}\Big)^{\frac{1}{2}}\Big)^*(\frac{\lambda|Q|}{16})\\
&&\qquad+\Big[\Big(\int^{\infty}_0\Big|\sum_{s=m+1}^{\infty}\sum_{j\in\mathbb{Z}}L_{j,\,t}*\big(B_{j-s}\chi_{\mathbb{R}^n\backslash 3Q}\big)\big|^2\frac{dt}{t}\Big)^{\frac{1}{2}}\Big]^*(\frac{\lambda|Q|}{16}).
\end{eqnarray*}
For eaxh $\xi\in Q$, write
\begin{eqnarray*}&&\int^{\infty}_0\Big|\sum_{s=1}^m\sum_{j\in\mathbb{Z}}L_{j,\,t}*\big(B_{j-s}\chi_{\mathbb{R}^n\backslash 3Q}\big)(\xi)\big|^2\frac{dt}{t}\\
&&\quad\lesssim \sup_{t>0}\big|\sum_{s=1}^m\sum_{j\in\mathbb{Z}}L_{j,\,t}*\big(B_{j-s}\chi_{\mathbb{R}^n\backslash 3Q}\big)(\xi)\big|\\
&&\qquad\times \int^{\infty}_0\Big|\sum_{s=1}^m\sum_{j\in\mathbb{Z}}L_{j,\,t}*\big(B_{j-s}\chi_{\mathbb{R}^n\backslash 3Q}\big)(\xi)\big|\frac{dt}{t}
\end{eqnarray*}
A trivial computation leads to that
\begin{eqnarray*}
&&\sup_{\xi\in Q}\sup_{t>0}\big|\sum_{s=1}^m\sum_{j\in\mathbb{Z}}L_{j,\,t}*\big(B_{j-s}\chi_{\mathbb{R}^n\backslash 3Q}\big)(\xi)\big|\\
&&\quad\lesssim
\sup_{\xi\in Q}\sup_{t>0}\sum_{j\in\mathbb{Z}}|L_{j,t}|*\Big(\sum_{s=1}^m|B_{j-s}|\Big)(\xi)\lesssim\inf_{y\in Q}
Mb(y).
\end{eqnarray*}
Let $T_j$ be the operator defined by
$$T_jh(x)=\int_{2^{j-1}\leq |x-y|\leq 2^{j+2}}\frac{|\Omega(x-y)|}{|x-y|^n}h(y)dy.
$$
We have
\begin{eqnarray*}
&&\int^{\infty}_0\Big|\sum_{s=1}^m\sum_{j\in\mathbb{Z}}L_{j,\,t}*\big(B_{j-s}\chi_{\mathbb{R}^n\backslash 3Q}\big)(\xi)\big|\frac{dt}{t}\lesssim \sum_{s=1}^m\sum_{j\in\mathbb{Z}}\mathcal{M}_{T_j}B_{j-s}(x).
\end{eqnarray*}
Therefore, for $x\in \mathbb{R}^n\backslash E^*$,
\begin{eqnarray*}
M_{\lambda/8,\widetilde{\mu}_{\Omega}}^{\mathcal{F}}b(x)\lesssim \Big[ Mb(x)\Big(\sum_{s=1}^m\sum_{j\in\mathbb{Z}}\mathcal{M}_{T_j}B_{j-s}(x)\Big)\Big]^{\frac{1}{2}}
+\mathcal{M}_{\lambda}b(x),
 \end{eqnarray*}
where and in the following,
$$\mathcal{M}_{\lambda}b(x)=\max_{Q\ni x,\,Q\in \mathcal{F}}\Big[\Big(\int^{\infty}_0\Big|\sum_{s=m+1}^{\infty}\sum_{j\in\mathbb{Z}}L_{j,\,t}*\big(B_{j-s}\chi_{\mathbb{R}^n\backslash 3Q}\big)\big|^2\frac{dt}{t}\Big)^{\frac{1}{2}}\Big]^*(\frac{\lambda|Q|}{16})$$
if $x\in\cup_{Q\in\mathcal{F}}Q$ and $M_{\lambda}b(x)=0$ otherwise. Lemma 2.6 in \cite{ler4} tells us that
 $$\|\mathcal{M}_{T_j}h\|_{L^1(\mathbb{R}^n)}\lesssim \|h\|_{L^1(\mathbb{R}^n)}.$$
Therefore,
\begin{eqnarray}\label{eq2.19}&&\big|\big\{x\in\mathbb{R}^n:\,
\Big[Mb(x)\Big(\sum_{s=1}^m\sum_{j\in\mathbb{Z}}\mathcal{M}_{T_j}B_{j-s}(x)\Big)\Big]^{\frac{1}{2}}>\frac{\alpha}{8}\big\}\big|\\
&&\quad\lesssim \big|\big\{x\in\mathbb{R}^n:\,Mb(x)>m^{-\frac{1}{2}}\frac{\alpha}{8}\big\}\big|\nonumber\\
&&\qquad+\big|\big\{x\in\mathbb{R}^n:\,\sum_{s=1}^m\sum_{j\in\mathbb{Z}}\mathcal{M}_{T_j}B_{j-s}(x)
>\frac{m^{\frac{1}{2}}\alpha}{8}\big\}\big|\nonumber\\
&&\quad\lesssim m^{\frac{1}{2}}\alpha^{-1}\|f\|_{L^1(\mathbb{R}^n)}+m^{-\frac{1}{2}}\alpha^{-1}\sum_{s=1}^m\big\|
\sum_{j\in\mathbb{Z}}\mathcal{M}_{T_j}B_{j-s}\big\|_{L^1(\mathbb{R}^n)}\nonumber\\
&&\quad\lesssim m^{\frac{1}{2}}\alpha^{-1}\|f\|_{L^1(\mathbb{R}^n)}.\nonumber
\end{eqnarray}

It remains to estimate $\mathcal{M}_{\lambda}$. Let $\zeta\in C^{\infty}(\mathbb{R})$ such that ${\rm supp}\,\zeta\subset \{r:\,|r|<2^{-10}\}$, and $\int\zeta(r)dr=1$. Set
$$K_{j,\,s;\,t}(x)=\chi_{(1,\,2]}(2^{-j}t)\frac{\Omega(x)}{|x|^{n}}v^s(\frac{|x|}{t}\big),
$$
with
$$v^s(r)=\int_{\mathbb{R}}h(r-u)2^{\beta s}\zeta(2^{\beta s}u)du,
$$
$\beta$ is a small constant which will be chosen later. By \cite[Lemma 3 and Lemma 4]{fans}, we know that there exists functions $\{M_{j,\,s;\,t}\}$ such that for some $\delta\in (0,\,1)$,
\begin{eqnarray*}
\Big\|\Big(\int^{\infty}_0\Big|\sum_{j}M_{j,\,s;\,t}*B_{j-s}\Big|^2\frac{dt}{t}\Big)^{\frac{1}{2}}\Big\|_{L^2(\mathbb{R}^n)}\lesssim 2^{-\delta s}\lambda^{-1}\alpha\sum_{Q}\|b_Q\|_{L^1(\mathbb{R}^n)},
\end{eqnarray*}
and for any $N\in\mathbb{N}$ and $\varepsilon\in (0,\,1)$.
\begin{eqnarray*}
&&\Big\|\Big(\int^{\infty}_0\Big|\sum_{j}\big(K_{j,\,s;\,t}-M_{j,\,s;\,t}\big)*B_{j-s}\Big|^2\frac{dt}{t}\Big)^{\frac{1}{2}}\Big\|_{L^1(\mathbb{R}^n)}\\
&&\quad\lesssim 2^{\beta s N}\big(2^{-\delta s}+2^{\beta sN}2^{(n+(\varepsilon+\delta-1)N\big)s}
\sum_{Q}\|b_Q\|_{L^1(\mathbb{R}^n)},\nonumber
\end{eqnarray*}
If we choose $\beta$ small enough, we then deduce from the last  two inequalities  that for a constant $\gamma\in(0,\,1)$
\begin{eqnarray}\label{eq2.20}
&&\Big|\big\{x\in\mathbb{R}^n:\, \Big(\int^{\infty}_0\Big|\sum_{s=m+1}^{\infty}\sum_{j\in\mathbb{Z}}K_{j,\,s;\,t}*B_{j-s}(x)\Big|^2\frac{dt}{t}
\Big)^{\frac{1}{2}}>\frac{\alpha}{32}\Big\}\Big|\\
&&\quad\lesssim\alpha^{-1}\lambda^{-1}\sum_{s=m+1}^{\infty}2^{-\gamma s}\|f\|_{L^1(\mathbb{R}^n)}\nonumber
\end{eqnarray}
Also, it was proved in \cite[p. 276]{fans} that
\begin{eqnarray}
&&\Big|\big\{x\in\mathbb{R}^n:\, \Big(\int^{\infty}_0\Big|\sum_{s=m+1}^{\infty}\sum_{j\in\mathbb{Z}}(L_{j,\,t}-K_{j,\,s;\,t}\big)*B_{j-s}(x)\Big|^2\frac{dt}{t}
\Big)^{\frac{1}{2}}>\frac{\alpha}{32}\Big\}\Big|\\
&&\quad\lesssim\alpha^{-1}\sum_{s=m+1}^{\infty}2^{-\beta s}\|f\|_{L^1(\mathbb{R}^n)}.\nonumber
\end{eqnarray}
This, along with (\ref{eq2.20}), implies that for a constant $\varrho>0$
\begin{eqnarray}\label{eq2.22}
&&\Big|\big\{x\in\mathbb{R}^n:\, \Big(\int^{\infty}_0\Big|\sum_{s=m+1}^{\infty}\sum_{j\in\mathbb{Z}}L_{j,\,t}*B_{j-s}(x)\Big|^2\frac{dt}{t}
\Big)^{\frac{1}{2}}>\frac{\alpha}{16}\Big\}\Big|\\
&&\qquad\lesssim\alpha^{-1}\lambda^{-1}2^{-\varrho m}\|f\|_{L^1(\mathbb{R}^n)}.\nonumber
\end{eqnarray}

We can now conclude the proof of (\ref{eq2.18}). As in \cite[Section 2]{ler5}, we can write
$$\big|\{x\in\mathbb{R}^n:\, \mathcal{M}_{\lambda}b(x)>\frac{\alpha}{8}\}\big|\leq \sum_{i\in\Lambda_1}|Q_i|+\sum_{j\in\Lambda_2}|I_j|,
$$
where for each $i\in\Lambda_1$,
$$|Q_i|<\frac{16}{\lambda}\Big|\Big\{x\in Q_i:\,\Big(\int^{\infty}_0\Big|\sum_{s=m+1}^{\infty}\sum_{j\in\mathbb{Z}}L_{j,\,t}*B_{j-s}(x)\Big|^2\frac{dt}{t}
\Big)^{\frac{1}{2}}>\frac{\alpha}{16}\Big\}\Big|,
$$
while for each $i\in \Lambda_2$,
$$|Q_i|<\frac{16}{\lambda}\Big|\Big\{x\in Q_i:\,\Big(\int^{\infty}_0\Big|\sum_{s=m+1}^{\infty}\sum_{j\in\mathbb{Z}}L_{j,\,t}*(B_{j-s}\chi_{3Q_i})(x)\Big|^2\frac{dt}{t}
\Big)^{\frac{1}{2}}>\frac{\alpha}{16}\Big\}\Big|,
$$
It follows from estimate (\ref{eq2.22}) that
\begin{eqnarray}\label{eq2.23}
\sum_{i\in\Lambda_1}|Q_i|\lesssim \alpha^{-1}\lambda^{-2}2^{-\varrho m}\|f\|_{L^1(\mathbb{R}^n)}.
\end{eqnarray}
On the other hand, as it was pointed out in \cite{ler5}, we have that for each fixed $j\in \Lambda_2$,
\begin{eqnarray*}|I_j|&\leq &\frac{16}{\lambda}\Big|\Big\{x\in I_j:\,\Big(\int^{\infty}_0\Big|\sum_{s=m+1}^{\infty}\sum_{j\in\mathbb{Z}}L_{j,\,t}*B^i_{j-s}(x)\Big|^2\frac{dt}{t}
\Big)^{\frac{1}{2}}>\frac{\alpha}{16}\Big\}\Big|\\
&\lesssim&\alpha^{-1}\lambda^{-2}2^{-\varrho m}\int_{3I_j}|f(x)|dx,\nonumber
\end{eqnarray*}
if $m\geq N_0$, with $N_0$ a large positive integer depending only $n$. Therefore, for a positive constant $C_n$, we have that
\begin{eqnarray*}
\sum_{j\in\Lambda_2}|I_j|\leq|\{x\in\mathbb{R}^n:\,Mf(x)>C_n\alpha^{}\lambda^{2}2^{\varrho m}\}|\lesssim \alpha^{-1}\lambda^{-2}2^{-\varrho m}\int_{\mathbb{R}^n}|f(x)|dx,
\end{eqnarray*}
which, along with (\ref{eq2.23}), implies that
\begin{eqnarray}\label{eq2.24}\big|\{x\in\mathbb{R}^n:\, \mathcal{M}_{\lambda}b(x)>\frac{\alpha}{8}\}\big|\lesssim\alpha^{-1}\lambda^{-2}2^{-\varrho m}\|f\|_{L^1(\mathbb{R}^n)}.
\end{eqnarray}
Take $m\in\mathbb{N}$ such that  $2N_0\big(\log(\frac{1}{\lambda})+1\big)\leq m\varrho<2N_0\big(\log(\frac{1}{\lambda})+1\big)+1$. (\ref{eq2.18}) now follows from (\ref{eq2.19}) and (\ref{eq2.24}) directly.
\end{proof}
Let $T$ be an operator and $r\in [1,\,\,\infty)$. Define the maximal operator $\mathscr{M}_{r,\,T}$ by
$$\mathscr{M}_{r,\,T}f(x)=\sup_{Q\ni x}\Big(\frac{1}{|Q|}\int_Q|T(f\chi_{\mathbb{R}^d\backslash 3Q})(\xi)|^rd\xi\Big)^{1/r}.$$
$\mathscr{M}_{r,\,T}$ was introduced by Lerner \cite{ler4} and is useful in establishing bilinear sparse domination of rough operator $T_{\Omega}$.
By Lemma 3.3 in \cite{ler4},  Theorem \ref{dingli2.1} implies  that for $p\in (1,\,\infty)$,
$$\|\mathscr{M}_{p,\,\widetilde{\mu}_{\Omega}}f\|_{L^{1,\,\infty}(\mathbb{R}^n)}\lesssim \|\Omega\|_{L^{\infty}(S^{n-1})}p^{\frac{1}{2}}\|f\|_{L^1(\mathbb{R}^n)}.
$$
Now we define the maximal operator $\mathscr{M}_{r,\,\widetilde{\mu}_{\Omega}}^* $ by $$\mathscr{M}_{r,\,\widetilde{\mu}_{\Omega}}^*f(x)=\sup_{Q\ni x}\Big(\frac{1}{|Q|}\int_{Q}\Big(\int^{\infty}_{\ell(Q)}|\widetilde{F}_tf(\xi)|^2\frac{dt}{t^3}\Big)^{\frac{r}{2}}d\xi\Big)^{\frac{1}{r}},$$
A trivial computation yields that
\begin{eqnarray*}
\int^{\infty}_{\ell(Q)}|\widetilde{F}_tf(\xi)|^2\frac{dt}{t^3}&=&
\int^{\infty}_{6n\ell(Q)}|\widetilde{F}_tf(\xi)|^2\frac{dt}{t^3}\\
&&+\int^{6n\ell(Q)}_{\ell(Q)}|\widetilde{F}_tf(\xi)|^2\frac{dt}{t^3}\\
&\leq &\int^{\infty}_{6n\ell(Q)}|\widetilde{F}_t(f\chi_{\mathbb{R}^n\backslash 3Q})(\xi)|^2\frac{dt}{t^3}+\inf_ {y\in Q}Mf(y).
\end{eqnarray*}
We then have the follwing conclusion.
\begin{corollary}\label{c2.1}Let $\Omega$ be homogeneous of degree zero and have mean value zero.
Suppose that $\Omega\in L^{\infty}(S^{n-1})$. Then for $r\in (1,\,\infty)$,
$$\|\mathscr{M}_{r,\,\widetilde{\mu}_{\Omega}}^*f\|_{L^{1,\,\infty}(\mathbb{R}^n)}\lesssim r^{\frac{1}{2}}\|\Omega\|_{L^{\infty}(S^{n-1})}\|f\|_{L^{1}(\mathbb{R}^n)}.$$
\end{corollary}

{\it Proof of Theorem \ref{dingli1.3}}. We employ the ideas used in \cite{ler5} and assume that $\|\Omega\|_{L^{\infty}(S^{n-1})}=1$. By (\ref{eq2.4}), it suffices to prove (\ref{eq:1.3}) with $\mu_{\Omega}$ replaced by $\widetilde{\mu}_{\Omega}$. For simplicity, we only prove (\ref{eq:1.3}) for the case $s=2$. The case $s\in [1,\,2)$ can be proved in the same way. For a fixed $r\in (1,\,2)$ and cube $Q_0$, define $\mathscr{M}_{r',\,\widetilde{\mu}_{\Omega};\,Q_0}$, the  local analogy of $\mathscr{M}_{r',\,\widetilde{\mu}_{\Omega}}$, as
$$ \mathscr{M}_{r',\,\widetilde{\mu}_{\Omega};\,Q_0}^*f(x)=\sup_{Q\ni x,\, Q\subset Q_0}\Big(\frac{1}{|Q|}\int_{Q}\Big(\int^{\infty}_{\ell(Q)}|\widetilde{F}_t(f\chi_{3Q_0})(\xi)|^2\frac{dt}{t^3}
\Big)^{\frac{r}{2}}d\xi\Big)^{\frac{1}{r}}.$$
Let $E=\cup_{j=1}^3E_j$ with
$$E_1=\big\{x\in Q_0:\, \widetilde{\mu}_{\Omega}(f\chi_{3Q_0})(x)|>D\langle |f|\rangle_{3Q_0}\big\},$$
$$E_3=\{x\in Q_0:\,\mathscr{M}^*_{r',\,\widetilde{\mu}_{\Omega};\,Q_0}(f\chi_{3Q_0})(x)>Dr'^{\frac{1}{2}}\langle |f|\rangle_{3Q_0}\},$$
where $D$  is a positive constant to be determined. If we choose $D$ large enough, it then follows from Corollary \ref{c2.1} and the weak type (1,\,1) boundedness of $\mu_{\Omega}$ (see \cite{fans}) that
$$|E|\le \frac{1}{2^{d+2}}|Q_0|.$$
Now on the cube $Q_0$, applying the Calder\'on-Zygmund decomposition to $\chi_{E}$ at level $\frac{1}{2^{d+1}}$,  we obtain pairwise disjoint cubes $\{P_j\}\subset \mathcal{D}(Q_0)$, such that
$$\frac{1}{2^{d+1}}|P_j|\leq |P_j\cap E|\leq \frac{1}{2}|P_j|$$
and $|E\backslash\cup_jP_j|=0$.  Observe that $\sum_j|P_j|\leq \frac{1}{2}|Q_0|$.
Write
\begin{equation*}
\int_{Q_0}|g(x)|\big[\widetilde{\mu}_{\Omega}(f\chi_{3Q_0})(x)\big]^2dx=\sum_{i=1}^2J_i,
\end{equation*}
where
$$
J_1=\int_{Q_0\backslash \cup_jP_j}|g(x)|\big[\widetilde{\mu}_{\Omega}(f\chi_{3Q_0})(x)\big]^2dx,\,\,
J_2=\sum_l\int_{P_l}|g(x)|\big[\widetilde{\mu}_{\Omega}(f\chi_{3Q_0})(x)\big]^2dx.
$$
The facts that $|E\backslash\cup_jP_j|=0$ implies that
$$|{\rm J}_1|\lesssim \langle|f|\rangle_{3Q_0}^2\langle |g|\rangle_{Q_0}|Q_0|.$$
To estimate $J_2$, write
\begin{eqnarray*}
\int_{P_l}|g(x)|\big[\widetilde{\mu}_{\Omega}(f\chi_{3Q_0})(x)\big]^2dx
&=&\int_{P_l}|g(x)|\int_{0}^{\ell(P_l)}|\widetilde{F}_{t}(f\chi_{3Q_0})(x)|^2\frac{dt}{t^3}dx\\
&&+\int_{P_l}|g(x)|\int_{\ell(P_l)}^{\infty}|\widetilde{F}_{t}(f\chi_{3Q_0})(x)|^2\frac{dt}{t^3}dx\\
&:=&J_{21}^l+J^l_{22}.
\end{eqnarray*}
It is obvious that
$$J^l_{21}=\int_{P_l}|g(x)|\int_{0}^{\ell(P_l)}|\widetilde{F}_{t}(f\chi_{3P_l})(x)|^2\frac{dt}{t^3}dx\leq\int_{P_l}|g(x)|\big[\widetilde{\mu}_{\Omega}(f\chi_{3P_l})(x)\big]^2dx.$$
The fact  that $P_l\cap E^c\not =\emptyset$ tells us that
\begin{eqnarray*}|J_{22}^l|&=&\int_{P_l}|g(x)|\int_{\ell(P_l)}^{\infty}|\widetilde{F}_{t}(f\chi_{3Q_0})(x)|^2\frac{dt}{t^3}dx\\
&\leq &\Big(\int_{P_l}|g(x)|^rdx\Big)^{\frac{1}{r}}\Big(\int_{P_l}\Big(\int^{\infty}_{\ell(P_l)}|\widetilde{F}_{t}(f\chi_{3Q_0})(x)|^2\frac{dt}{t^3}\Big)^{r'}\Big)^{\frac{1}{r'}}\\
&\leq&|P_l|\langle |g|\rangle_{P_l,\,r}\inf_{y\in P_l}\big(\mathscr{M}^*_{2r',\,\widetilde{\mu}_{\Omega};\,Q_0}f(y)\big)^2,
\end{eqnarray*}
and so
$$\sum|J_{22}^l|\leq r'\langle |f|\rangle_{3Q_0}^2\sum_l|P_l|\langle |g|\rangle_{P_l,\,r}\leq r'\langle |f|\rangle_{3Q_0}^2 |g|\rangle_{Q_0,\,r}|Q_0|.
$$
Combining  estimates for $J_1$, $J_{21}^l$ and $J_{22}^l$ shows that
\begin{eqnarray}\label{eq:local1}
\int_{Q_0}|g(x)|\big[\widetilde{\mu}_{\Omega}(f\chi_{3Q_0})(x)\big]^2dx&\leq& Cr'\langle |f|\rangle_{3Q_0}^2 |g|\rangle_{Q_0,\,r}|Q_0|\\
&&+
\sum_l\int_{P_l}|g(x)|\big[\widetilde{\mu}_{\Omega}(f\chi_{3P_l})(x)\big]^2dx.\nonumber
\end{eqnarray}

Recall that  $\sum_j|P_j|\leq \frac{1}{2}|Q_0|$. Iterating   estimate (\ref{eq:local1}), we obtain that there exists a $\frac{1}{2}$-sparse family of cubes $\mathcal{F}\subset \mathcal{D}(Q_0)$, such that
$$\int_{Q_0}|g(x)|\big[\widetilde{\mu}_{\Omega}(f\chi_{3Q_0})(x)\big]^2dx\lesssim r'\sum_{Q\in\mathcal{F}}\langle|f|\rangle_{3Q}^2\langle |g|\rangle_{Q,\,r}|Q|;$$
see also \cite{ler4,ler5}.

We can now conclude the proof of Theorem \ref{dingli1.3}. In fact, as in \cite{ler5}, we decompose $\mathbb{R}^n$ by cubes $\{R_l\}$, such that ${\rm supp}f\subset 9R_l$ for each $l$, and $R_l$'s have disjoint interiors.
Then for each $l$, we have a $\frac{1}{2}$-sparse family of cubes $\mathcal{F}_l\subset \mathcal{D}(R_l)$, such that
\begin{eqnarray*}
\int_{R_l}|g(x)|\big[\widetilde{\mu}_{\Omega}(f\chi_{3R_l})(x)\big]^2dx\leq Cr'\sum_{Q\in\mathcal{F}_l}\langle |f|\rangle_{3Q}^2 \langle|g|\rangle_{Q,\,r}|Q|.
\end{eqnarray*}
Let $\mathcal{S}=\cup_l\{3Q:\, Q\in \mathcal{F}_l\}$. Summing over the last inequality yields our desired conclusion.\qed
\section{Proof of Theorem \ref{dingli1.4} and Theorem \ref{dingli1.5}}
Let $\mathcal{S}$ be a sparse family of cubes. Associated with  the sparse family $\mathcal{S}$ and $r\in (0,\,\infty)$, we define the sparse operator
$\mathcal{A}_{\mathcal{S}}^{r}$  by
$$\mathcal{A}_{\mathcal{S}}^{r}f(x)=\Big\{\sum_{Q\in\mathcal{S}}\big(\langle |f|\rangle_{Q}\big)^r\chi_{Q}(x)\Big\}^{1/r}.$$

\begin{lemma} \label{lem3.2}Let $r\in (0,\,\infty)$,  and $\mathcal{S}$ be a sparse family of cubes. Then
\begin{itemize}
\item[\rm (i)] for $p\in (1,\,\infty)$ and $w\in A_{p}(\mathbb{R}^n)$,$$\|\mathcal{A}_{\mathcal{S}}^{r}f\|_{L^p(\mathbb{R}^n,\,w)}\lesssim [w]_{A_p}^{\frac{1}{p}}
\big([w]_{A_{\infty}}^{(\frac{1}{r}-\frac{1}{p})_+}+[w^{1-p'}]_{A_{\infty}}^{\frac{1}{p}}\big)
\|f\|_{L^{p}(\mathbb{R}^n,\,w)};$$
\item[\rm (ii)] for $p\in (r,\,\infty)$ and $w\in A_{p}(\mathbb{R}^n)$,
\begin{eqnarray}&&\sum_{Q\in\mathcal{S}}\langle |f|w^{1-p'}\rangle_{Q}^r\langle|g|w\rangle_{Q}|Q|\\
&&\quad\lesssim [w]_{A_p}^{\frac{r}{p}}\big([w]_{A_{\infty}}^{1-\frac{r}{p}}+[w^{1-p'}]_{A_{\infty}}^{\frac{r}{p}}\big)
\|f\|_{L^p(\mathbb{R}^n,\,w^{1-p'})}^r\|g\|_{L^{(p/r)'}(\mathbb{R}^n,\,w)}.\nonumber\end{eqnarray}
\end{itemize}
\end{lemma}
Conclusion (i) of Lemma \ref{lem3.2} is a combination of Theorem 2.3 and Theorem 2.13 in \cite{lali}, and conclusion (ii) of Lemma \ref{lem3.2} was proved in the proof of Theorem 2.3 in \cite{lali}.
\begin{lemma}\label{lem3.3} Let $w\in A_{\infty}(\mathbb{R}^n)$. Then for any cube $Q$ and $\delta\in (1,\,1+\frac{1}{2^{11+n}[w]_{A_{\infty}}}]$,
$$\Big(\frac{1}{|Q|}\int_Qw^{\delta}(x)dx\Big)^{\frac{1}{\delta}}\leq \frac{2}{|Q|}\int_{Q}w(x)dx.$$
\end{lemma}
Lemma \ref{lem3.2} was proved in \cite{hp}.
\begin{lemma}\label{lem3.4}Let  $\alpha,\,\beta\in \mathbb{N}\cup\{0\}$ and $U$ be a  sublinear operator. Suppose that for any  $r\in (1,\,2)$, and a bounded function $f$ with compact support, there exists a sparse family of cubes $\mathcal{S}$, such that for any function $g\in L^1(\mathbb{R}^d)$,
\begin{eqnarray}\label{en.sparse1}\Big|\int_{\mathbb{R}^d}Uf(x)g(x)dx\Big|\leq r'^{\alpha}\mathcal{A}_{\mathcal{S};\,L(\log L)^{\beta},\,L^r}(f,\,g).\end{eqnarray}
Then for any $w\in A_1(\mathbb{R}^d)$ and bounded function $f$ with compact support,
\begin{equation*}
\begin{split}
w(\{&x\in\mathbb{R}^d:\, |Uf(x)|>\lambda\})\\
&\quad\lesssim [w]_{A_{\infty}}^\alpha\log^{1+\beta}({\rm e}+[w]_{A_{\infty}})[w]_{A_1}\int_{\mathbb{R}^d}\frac{|f(x)|}{\lambda}\log ^{\beta}\Big({\rm e}+\frac{|f(x)|}{\lambda}\Big)w(x)dx.
\end{split}
\end{equation*}
\end{lemma}

This was proved Hu, Lai and Xue \cite{hulai}, see also Appendix C in \cite{lpr} for the case of $\beta=0$.

{\it Proof of Theorem \ref{dingli1.4}}.   We assume that $\|\Omega\|_{L^{\infty}(S^{n-1})}=1$. Let $p\in (1,\,\infty)$ and $w\in A_p(\mathbb{R}^n)$. Set $\tau_w=1+\frac{1}{2^{11+n}[w]_{A_{\infty}}}$. For a bounded function $f$ with compact support, it follows from Theorem 1.3 and Lemma \ref{lem3.3} that
\begin{eqnarray*}
\int_{\mathbb{R}^n}\big(\mu_{\Omega}(f)(x)\big)^pw(x)dx&\lesssim&\tau_w'^{\frac{p}{2}}\sum_{Q\in\mathcal{S}}\langle |f|\rangle_Q^{p}\langle w\rangle_{Q,\,\tau_w}|Q|\\
&\lesssim&[w]_{A_{\infty}}^{\frac{p}{2}}\sum_{Q\in\mathcal{S}}\langle |f|\rangle_Q^{p}w(Q)\\
&=&[w]_{A_{\infty}}^{\frac{p}{2}}\int_{\mathbb{R}^n}\big(\mathcal{A}_{\mathcal{S}}f(x)\big)^pw(x)dx\\
&\lesssim&[w]_{A_p}
[w^{1-p'}]_{A_{\infty}}[w]_{A_{\infty}}^{\frac{p}{2}}\|f\|_{L^p(\mathbb{R}^n,\,w)}^p.
\end{eqnarray*}
This establishes (\ref{eq:1.4}) for $p\in (1,\,2)$.

We now prove (\ref{eq:1.4}) for the case of $p\in (2,\,\infty)$. Let $f$ be a bounded function with compact support, $p\in (1,\,\infty)$ and $w\in A_p(\mathbb{R}^n)$. Set $r=1+\frac{1}{4p\tau_{w}}$ and $s=1+\frac{1}{2p}$. Then $(r-1/s)s'<\tau_w$ and $rs<1+\frac{5}{6p}<(p/2)'$. By Theorem \ref{dingli1.3}, there exists a sparse family of cubes $\mathcal{S}$, such that for any $g\in L_{{\rm loc}}^r(\mathbb{R}^n)$, such that
\begin{eqnarray}\label{eq3.3}
\int_{\mathbb{R}^n}\big(\mu_{\Omega}(f)(x)\big)^2g(x)dx\lesssim r'\sum_{Q\in\mathcal{S}}\langle |f|\rangle_{Q}^2\langle |g|\rangle_{Q,\,r}|Q|.
\end{eqnarray}
By Lemma 2.1, we may assume that $\mathcal{S}\subset \mathscr{D}$ with $\mathscr{D}$ a dyadic lattice. A sraightforward computation involving H\"older's inequality, Lemma \ref{lem3.2} and Lemma \ref{lem3.3} shows that
\begin{eqnarray*}
\langle|g|w\rangle_{Q,\,r}&\leq& \langle |g|^{sr}w\rangle_Q^{\frac{1}{sr}}\langle w^{(r-\frac{1}{s})s'}\rangle_{Q}^{\frac{1}{rs'}}\\
&\lesssim&\langle |g|\rangle_{sr,\,Q}^w\langle w\rangle_Q^{1-\frac{1}{sr}}\lesssim\langle M_{sr}^{\mathscr{D},\,w}g\rangle_{Q}^w \langle w\rangle_{Q},
\end{eqnarray*}
where $$\langle h\rangle_{Q}^w=\frac{1}{w(Q)}\int_{Q}h(y)w(y),$$
and
$$M_{sr}^{\mathscr{D},\,w}h(y)=\sup_{Q\ni y,\,Q\in\mathscr{D}}\Big(\frac{1}{w(Q)}\int_{Q}|h(z)|^{sr}w(z)dz\Big)^{\frac{1}{sr}}.
$$
Thus by (\ref{eq3.3}),
\begin{eqnarray*}
\int_{\mathbb{R}^n}\big[\mu_{\Omega}(f w^{1-p'})(x)\big]^2|g(x)|w(x)dx&\lesssim &[w]_{A_{\infty}}\sum_{Q\in\mathcal{S}}\langle |f|w^{1-p'}\rangle_{Q}^2\langle |g|w\rangle_{Q,\,r}|Q|\\
&\lesssim&[w]_{A_{\infty}}\sum_{Q\in\mathcal{S}}\langle |f| w^{1-p'}\rangle_{Q}\langle  M_{sr}^{\mathscr{D},\,w}g\rangle_{Q}^ww(Q).
\end{eqnarray*}
Recall that $rs<(p/2)'$ and $M_{sr}^{\mathscr{D},\,w}$ is bounded on $L^{(\frac{p}{2})'}(\mathbb{R}^n,\,w)$ with bound depending only on $n$ and $p$.
We deduce from Lemma \ref{lem3.2} that
\begin{eqnarray*}
&&\int_{\mathbb{R}^n}\big[\mu_{\Omega}(f w^{1-p'})(x)\big]^2|g(x)|dx\\
&&\quad\lesssim [w]_{A_p}^{\frac{2}{p}}\big([w]_{A_{\infty}}^{1-\frac{2}{p}}+[w^{1-p'}]_{A_{\infty}}^{\frac{2}{p}}\big) [w]_{A_{\infty}}\|f\|_{L^p(\mathbb{R}^n,\,w^{1-p'})}^2\|M_{sr}^{\mathscr{D},\,w}g\|_{L^{(\frac{p}{2})'}(\mathbb{R}^n,\,w)}\\
&&\quad\lesssim [w]_{A_p}^{\frac{2}{p}}\big([w]_{A_{\infty}}^{1-\frac{2}{p}}+[w^{1-p'}]_{A_{\infty}}^{\frac{2}{p}}\big) [w]_{A_{\infty}}\|f\|_{L^p(\mathbb{R}^n,\,w^{1-p'})}^2\|g\|_{L^{(\frac{p}{2})'}(\mathbb{R}^n,\,w)}.
\end{eqnarray*}
This in turn implies (\ref{eq:1.4}) for the case of $p\in (2,\,\infty)$.

Finally, (\ref{eq:1.5}) follows from Theorem \ref{dingli1.3} and Lemma \ref{lem3.4}. This completes the proof of Theorem \ref{dingli1.3}.\qed

{\it Proof of Theorem \ref{dingli1.5}}. We assume that $\|\Omega\|_{L^{\infty}(S^{n-1})}=1$. Let $w\in A_{\infty}(\mathbb{R}^n)$. Repeating the proof of Theorem 1.1 in \cite{lpr}, we deduce from Theorem \ref{dingli1.3} that
for bounded function  $f$ with compact support,
$$\int_{\mathbb{R}^n}\big(\mu_{\Omega}(f)(x)\big)^pw(x)dx\lesssim [w]^{\frac{p}{2}+1}_{A_{\infty}}\int_{\mathbb{R}^n}\big(Mf(x)\big)^pw(x)dx,\,\,\,\text{if} \,\,p\in [1,\,2],$$
and for bounded functions $f$ and $g$ with compact supports,
$$\int_{\mathbb{R}^n}\big(\mu_{\Omega}(f)(x)\big)^2g(x)w(x)dx\lesssim [w]^{2}_{A_{\infty}}
\|Mf\|_{L^p(\mathbb{R}^n,\,w)}^2\|g\|_{L^{(\frac{p}{2})'}(\mathbb{R}^n,\,w)},\,\,\text{if} \,\,p\in (2,\,\infty].$$
Our desired conclusion follows from the last two inequalities.
\qed
\medskip

{\bf Acknowledgement}.
The authors would like to thank Dr. Kangwei Li for his helpful suggestions concerning the weighted estimates for sparse operators.

\end{document}